\documentclass[12pt]{amsart}
\usepackage{graphicx}
\usepackage{amssymb}
\pdfoutput=1
\usepackage[center]{caption}
\usepackage{xy}
\xyoption{all}

\newcommand{\Z}{\mathbb Z}

\newcommand{\R}{\mathbb R}

\newcommand{\cp}{C_+}
\newcommand{\nun}[1]{\buildrel _{\circ} \over {\nu}(#1)}  
\newcommand{\la}{\langle}
\newcommand{\ra}{\rangle}

\def\tY{\tilde{Y}}

\newtheorem{theorem}{Theorem}[section]

\newtheorem{lemma}[theorem]{Lemma}

\newtheorem{proposition}[theorem]{Proposition}
\newtheorem{corollary}[theorem]{Corollary}%
\theoremstyle{definition}

\title{Smooth surfaces with non-simply-connected complements}
\author[Hee Jung Kim]{Hee Jung Kim}
\address{Department of Mathematics\newline\indent
Louisiana State University\newline\indent
Baton Rouge, LA 70803}
\email{\rm{heekim@math.lsu.edu}}
\author[Daniel Ruberman]{Daniel Ruberman}
\address{Department of Mathematics, MS 050\newline\indent Brandeis
University \newline\indent Waltham, MA 02454}
\email{\rm{ruberman@brandeis.edu}}
\begin{document}
\begin{abstract}
We give two constructions of surfaces in simply-connected $4$-manifolds with non simply-connected complements.  One is an iteration of the twisted rim surgery introduced by the first author~\cite{kim:surfaces}.  We also construct, for any group $G$ satisfying some simple conditions, a simply-connected symplectic manifold containing a symplectic surface whose complement has fundamental group $G$.  In  each case, we produce infinitely many smoothly inequivalent surfaces that are equivalent up to smooth s-cobordism and hence are topologically equivalent for good groups.
\end{abstract}
\maketitle

\section{Introduction}
In this paper, we study surfaces embedded in simply-connected $4$-manifolds whose complements are not simply-connected.  In the first part of the paper, we use a variation of the Fintushel-Stern rim surgery technique~\cite{fs:surfaces} introduced in first author's thesis~\cite{kim:surfaces}, called {\em $m$-twist} rim surgery.  Starting with an embedded surface $\Sigma \subset X$, and a knot $K$ in $S^3$,  $m$-twist rim surgery produces a new surface $\Sigma_K(m) \subset X$.  This construction shares with rim surgery the property that for suitable initial pairs (SW-pairs in the terminology of~\cite{fs:surfaces}) the resulting surface $(X,\Sigma_K(m))$ is smoothly knotted with respect to $(X,\Sigma)$.  For instance, this will be the case if $X$ is symplectic and $\Sigma$ symplectically embedded, and the Alexander polynomial of $K$ is nontrivial.   For some choices of the parameter $m$, this construction produces new knot groups of interest.   For example, for any odd number $n$, we get infinitely many knotted surfaces in $S^2\times S^2$ with knot group a dihedral group $D_{2n}$.

In some other circumstances, for instance if $m=1$, we show that $m$-twist rim surgery does not change the surface knot group $\pi_1(X-\Sigma)$.  We will show that in many cases, $(X,\Sigma_K(m))$ is topologically unknotted.   This generalizes our earlier paper~\cite{kim-ruberman:surfaces} which dealt with the case that $\pi_1(X-\Sigma)$ is finite cyclic.  These results rely on the $5$-dimensional s-cobordism theorem, which at present holds for a restricted class of groups~\cite{freedman-quinn,freedman-teichner:subexponential,krushkal-quinn:subexponential}, normally referred to as `good' groups.

In a somewhat different direction, we investigate the possibilities for the knot group of a symplectic surface in a simply-connected symplectic manifold.  We show that the obvious topological necessary conditions on a group $G$ are in fact sufficient to show that $G = \pi_1(X-\Sigma)$ where $\Sigma$ is a symplectic surface and $X$ is simply-connected .  These surfaces can be further modified by twisted rim surgery to produce infinite families of smoothly knotted surfaces.

\section{Twisted rim surgery}\label{rimsec}
Let $X$ be a simply-connected 4-manifold and $\Sigma$ an oriented
embedded surface in $X$.  Like the original rim surgery~\cite{fs:surfaces}, the operation of $m$-twist rim surgery from~\cite{kim:surfaces} provides a method to modify surfaces without changing the ambient $4$-manifold $X$.  The extra twist in the construction gives rise to some interesting surface knot groups.  Let us briefly review the construction. Let $K$ be any knot in $S^3$ and
$E(K)$ be its exterior. Consider a torus $T$ with $T\cdot T=0$,
called  a {\em rim torus}, which is the preimage in $\partial\nu(\Sigma)$
of a closed curve $\alpha$ in $\Sigma$.  A new
surface $\Sigma_K(m)$ is obtained by taking out a neighborhood $T\times D^2$ of a
rim torus from $X$ and gluing $S^1\times E(K)$ back
using an additional twist on the boundary. An equivalent
description of this construction is given in ~\cite{kim:surfaces}.
Identify the neighborhood $\nu(\alpha)$ of the curve $\alpha$ in $X$
with $S^1\times B^3$ so that the restriction of $\nu(\alpha)$ to
$\Sigma$ has the form $S^1\times I$. We now consider a self
diffeomorphism $\tau$ of $(S^3,K)$ called the `twist map' along $K$. Let $\partial E(K)\times I=K\times \partial D^2\times I$ be a collar of $\partial E(K)$ in $E(K)$ under a suitable trivialization with $0$-framing. The map $\tau$ is given by
\begin{equation}\label{tau}
\tau(\overline\theta, e^{i\varphi}, t) =(\overline\theta,
e^{i(\varphi + 2\pi t)}, t) \quad \mbox{for} \quad (\overline\theta,
e^{i\varphi}, t)\in K\times{\partial{D^2}}\times{I}
\end{equation}
and otherwise, $\tau(y)=y$. (Here we use $K\cong S^1\cong \R/\Z$.)

For any integer $m$, we
define the $m$-twist rim surgery on $(X,\Sigma)$ by an operation
producing a new pair
\begin{equation}\label{twistrimsurgery}
(X,\Sigma_K(m))=(X,\Sigma)-S^1\times (B^3,I)\cup_{\partial} S^1 \times_{\tau^m}(B^3,K_+).
\end{equation}
Here, we have written $(S^3,K) = (B^3,K_+) \cup  (B^3,K_{-})$ where $(B^3,K_{-})$ is an unknotted ball pair.  Note that $E(K)$ can be viewed as a codimension-$0$ submanifold of the complement $\cp(K) = B^3 -K_+$ onto which $\cp(K)$ deformation retracts.  Hence we can regard $\tau$ as an automorphism of the pair $(B^3,K_+)$, or equally as an automorphism of $\cp(K)$ that is the identity near $K_+$.
\subsection{Twisted rim surgery and the knot group}
For a surface $\Sigma$ carrying a non-trivial homology class in a simply-connected $4$-manifold $X$, the first homology group $H_1(X-\Sigma)$ is always finite cyclic, of order that we will usually write as $d$.  This coincides with the multiplicity of the homology class carried by $\Sigma$ in $H_2(X)$.   The way in which $m$-twist rim surgery affects the fundamental group of a surface knot depends to some degree on the relation between $m$ and $d$.  In this section, we assume $\pi_1(X-\Sigma)=\Z/d$, generated by the meridian $\mu_{\Sigma}$ of $\Sigma$.

In our previous paper~\cite{kim-ruberman:surfaces}, we considered the $m$-twist rim surgery in the case that $(m,d) =1$, and showed that the group $\pi_1(X-\Sigma_K(m))$ is $\Z/d$, no matter what $K$ is.  See Proposition~\ref{fundamentalgp} below for a generalization.    We now consider the opposite situation, in which $m=d$.    Using $m$-twist rim surgery along appropriate knots, we will construct surfaces in $X$ whose surface knot group are some non-abelian finite groups.   In what follows we will denote by $Y^d$ a cyclic $d$-fold cover of a space $Y$, and by $(Y,K)^d$ a $d$-fold cover of $Y$ branched along a submanifold $K$.

Take a curve $\alpha$ in $\Sigma$ and a framing of $\nu(\Sigma)$ along such that the push-off of $\alpha$ into
$\partial\nu(\Sigma)$ is homologically trivial in $X-\Sigma$.  For any knot $K$ in $S^3$, performing $d$-twist rim surgery along the rim torus $T\cong\alpha\times \mu_{\Sigma}$ gives a new surface
$\Sigma_K(d)$ in $X$ with $H_1(X-\Sigma_K(d))\cong\Z/d$.

\begin{lemma}\label{branchcover}
Suppose $\pi_1(X-\Sigma)\cong \Z/d$.  Then  $\pi_1(X-\Sigma_K(d))$ is a semi-direct product of $\pi_1((S^3,K)^d)$ and $\Z/d$, where the action of $\Z/d$ is by the covering transformations of the branched cover.
\end{lemma}

\begin{proof}
Write $H = \pi_1((S^3,K)^d)$.   It is sufficient to show that the fundamental group of the $d$-fold
unbranched cover $(X-\Sigma_K(d))^d$ is $H$, and that the exact sequence
\begin{equation}\label{hurewicz}
0\rightarrow H\rightarrow \pi_1(X-\Sigma_K(d))\stackrel{hurew}\longrightarrow\Z/d\rightarrow 0
\end{equation}
splits; the identification of the action of $\Z/d$ should be clear by the end of the argument.
Considering ~\eqref{twistrimsurgery} and the choice of the curve
$\alpha$, we decompose $(X-\Sigma_K(d))$ as
\begin{equation*}\label{complement}
X-\Sigma_K(d) =X-\Sigma - (S^1\times (B^3-I))\cup_{\partial} S^1
\times_{\tau^d} \cp(K)
\end{equation*}
with a corresponding decomposition for the $d$-fold cover:
$(X-\Sigma_K(d))^d$
\begin{equation}\label{branchcover1}
(X-\Sigma_K(d))^d =(X-\Sigma)^d - (S^1\times (B^3-I))\cup_{\partial} S^1
\times_{\tilde\tau^d} \cp(K)^d.
\end{equation}


Referring to the decomposition~\eqref{complement}, note that the inclusion of $X-\Sigma - (S^1\times (B^3-I))$ into $X-\Sigma$ induces an isomorphism on $\pi_1$, so the meridian $\mu_{\Sigma}$ has order $d$ in $ \pi_1(X-\Sigma_K(d))$.  It follows that the sequence~\eqref{hurewicz} splits, as asserted, and that the action of $\Z/d$ on the kernel of the Hurewicz map is given by conjugation by $\mu_{\Sigma}$.

Applying van Kampen's Theorem to the decomposition~\eqref{branchcover1} gives the
following diagram:
$$
\xymatrix@C=0pc{
&\pi_1((X-\Sigma)^d - S^1\times (B^3-I))Ê\ar[dr]^{\psi_1}&\\
{\pi_1 (S^1\times(\partial{B^3}-\{\mbox{two points}\}))} Ê\ar[ur]^{\varphi_1} \ar[dr]^{\varphi_2}
&&\pi_1((X - \Sigma_{K}(d))^d)\\
&\pi_1(S^1\times_{\tilde\tau^d} \cp(K)^d)Ê\ar[ur]^{\psi_2}&
}
$$


Note that $(X-\Sigma)^d - S^1\times (B^3-I)$ is isomorphic to
$(X-\Sigma)^d$ in $\pi_1$ which is trivial.  So, the diagram
shows
\begin{equation}\label{unbranchcover1}
\pi_1((X - \Sigma_{K}(d))^d)=\la \pi_1(E(K)^d,*) \mid \mu_{\tilde
K}=1, \beta =\tilde\tau_{*}^{d}(\beta),
\forall\beta\in\pi_1(E(K)^d,*)\ra
\end{equation}
 where $\mu_{\tilde K}$ is a
meridian of the lifted knot $K$.

Recall that the lift $\tilde\tau$ is given in ~\cite{kim:surfaces};
\begin{equation}\label{lifttau}
\tilde\tau (x) =
\begin{cases}
\phi(x) & \text{if $x\in {E(K)}^d-\partial {E(K)}^d\times I$}\\
(\overline{\theta}, e^{i((s/d)\cdot 2\pi +\varphi)},s) & \text{if
$x=(\overline{\theta}, e^{i{\varphi}},s)\in \partial {E(K)}^d\times
I$}\\
 x& \text{otherwise}
\end{cases}
\end{equation}
where $\phi$ is the canonical generator of the group $\Z/d$ of
covering transformations.

We observe that the lifted map $\tilde\tau^d$ is the same as a twist map
$\tau_{\tilde K}$ along the lifted knot of $K$ as in~\eqref{tau}
and so in the presentation ~\eqref{unbranchcover1}, we have that
$\tilde\tau_{*}^{d}(\beta)=\mu_{\tilde K}^{-1}\beta\mu_{\tilde K}$ for any
$\beta\in\pi_1(S^3,K,*)^d$. Since $\mu_{\tilde K}=1$,
$\tilde\tau_{*}^{d}(\beta)=\beta$. This implies

$$\pi_1((X - \Sigma_{K}(d))^d)=\la \pi_1(E(K)^d,*) \mid \mu_{\tilde
K}=1\ra=\pi_1((S^3,K)^d)=H.$$
\end{proof}

\begin{corollary}\label{finitebranchcover}
If $\pi_1(X- \Sigma) \cong \Z/d$ and $\pi_1((S^3,K)^d)$ is finite, then so is $\pi_1(X-\Sigma_{K}(d))$.
\end{corollary}
Since $\pi_1(X-\Sigma_{K}(d))$ is a semi-direct product of $H$ and
$\Z/d$, we denote $\pi_1(X-\Sigma_{K}(d))$ by $G$.

Corollary~\ref{finitebranchcover} suggests the question: what finite groups can be obtained by twisted rim surgery, starting with a surface whose knot group is cyclic?  From Perelman's work on geometrization~\cite{cao-zhu:ricci,morgan-tian:poincare} it suffices to know which spherical space forms arise as cyclic branched covers of knots in $S^3$.  The possibilities for the the fundamental groups of such branched covers, as well as the action of $\Z/d$, were determined by Plotnick and Suciu~\cite[Section 5]{plotnick-suciu:spherical}.  The full list is a little complicated, but the following are worth noting:
\begin{enumerate}
\item\label{quaternion} Taking $d=3$ and $K$ a trefoil knot then $H$ is a quaternion group $Q_8$, with the action of $\Z/3$ permuting the unit quaternions $\imath,\ \jmath$ and $k$.  So for example, starting with a degree-$3$ curve in ${\bf CP}^2$, we obtain an embedded torus in ${\bf CP}^2$ with group $G = Q(8) \rtimes \Z/3$.
\item\label{dihedral} Taking $d=2$ and $K$ to be a $2$-bridge knot $K_{p,q}$ with
$(p,q)=1$ then  $H$ is a cyclic group $\Z/p$, where $\Z/2$ acts by multiplication by $-1$, so that $G$
is a dihedral group $D_{2p}$.   This group can be realized as a surface knot group in ${\bf CP}^2$, by taking $\Sigma$ to be a degree-$2$ curve (a sphere) with a handle added to create a torus.  In the next section, we will want to perform a further twisted rim surgery on the resulting surface $\Sigma_{K}(2)$, but $({\bf CP}^2,\Sigma_{K}(2))$ is not an SW-pair in the sense of~\cite{fs:surfaces}, and so is not a good starting point for rim surgery constructions.  One could instead choose $\Sigma$ to be a curve in $S^2 \times S^2$ of bidegree $(2,2)$.
\item The Poincare homology sphere is the $p$ fold cover of the $(q,r)$ torus knot for $\{p,q,r\} = \{2,3,5\}$, giving three different extensions $G$ with subgroup $H =I^{*}=\pi_1(PHS)$.  For $d= 3,\ 5$ one obtains interesting surfaces in ${\bf CP}^2$, while for $d=2$ we would work with surfaces in $S^2 \times S^2$.
\end{enumerate}

A further interesting aspect of the second family of surfaces is that the group doesn't depend on $q$, but the knots $\Sigma_{K_{p,q}}(2)$ and  $\Sigma_{K_{p,q'}}(2)$ will be different if  $\Delta_{K_{p,q}}(t) \neq \Delta_{K_{p,q'}}(t)$ and $(X,\Sigma)$ is an SW-pair. So we can in principle obtain many knotted surfaces with a given dihedral knot group.  However, in the next section, we will do better than this, and obtain infinitely many such surfaces.

\subsection{Iterated rim surgery}\label{iterated}
We have constructed surfaces whose surface knot group is no longer abelian by $d$-twist rim surgery, starting with a surface whose complement has $\pi_1 = \Z/d$. Now, we seek to modify these
surfaces using twist rim surgery without changing the fundamental group.   With some additional hypotheses, this will produce surfaces that are smoothly knotted but topologically standard.

Given a sequence $K_1,\ldots,K_n$ of knots and integers $m_1,\ldots,m_n$, and a surface $\Sigma \subset X$, we can do a sequence of twisted rim surgeries. The result of this iterated rim surgery will be denoted $(X,\Sigma_{K_1,\ldots,K_n}(m_1,\ldots,m_n))$.  We will generally assume that the curves $\alpha_i \subset \Sigma$ that determine the rim tori are all parallel on $\Sigma$, and that curve $\alpha$ has a pushoff that is homotopically trivial in $X- \Sigma$.   It is easy to see that this condition is preserved after each surgery.

Fix an integer $d$, and start with a surface $\Sigma \subset X$ with $H_1(X -\Sigma) =\Z/d$, and such that the meridian $\mu_\Sigma$ has order $d$ in $G =\pi_1(X -\Sigma)$.  (This is the case for the surfaces constructed in the previous section.)   For any integer $m$ and knot $J$, consider the $m$-twist rim surgery along a rim torus $\alpha \times \mu_\Sigma$. Then in
certain cases, the knot group of the surface is preserved.
\begin{proposition}\label{fundamentalgp}
If $(m,d)=1$ then the knot group of $\Sigma_J(m)$ is isomorphic to $G =\pi_1(X -\Sigma)$.
\end{proposition}
\begin{proof} We first note that the Hurewicz homomorphism gives
$$
\pi_1(X-{\Sigma_J(m)})\rightarrow
H_1(X-{\Sigma}_J(m))=H_1(X-{\Sigma})=\Z/d.
$$
We shall show that the fundamental group of the $d$-fold cover
$(X-{\Sigma}_J(m))^d$ of $X-{\Sigma_J(m)}$ is isomorphic to
$\pi_1((X-{\Sigma})^d)$. Identifying a lift $\tilde\alpha$ of $\alpha$ in $d$-fold cover of $X$ branched along $\Sigma$ as $S^1$, we have a decomposition similar to that in~\eqref{branchcover1}:
\begin{equation*}\label{branchcover2}
(X-\Sigma_J(m))^d =(X-\Sigma)^d-S^1\times (B^3-I)\cup_{\partial}
S^1 \times_{\tilde\tau^m} \cp(J)^d
\end{equation*}

The van Kampen Theorem for this decomposition gives the following
diagram
\begin{equation}\label{d-diagram}
\xymatrix@C=-2pc{
&\pi_1((X-\Sigma)^d-S^1\times (B^3-I))Ê\ar[dr]^{\psi_1}&\\
{\pi_1 (S^1\times(\partial{B^3}-\{\mbox{two points}\}))} Ê\ar[ur]^{\varphi_1} \ar[dr]^{\varphi_2}
&& \pi_1((X - \Sigma_{J}(m))^d)\\
& \pi_1(S^1\times_{\tilde\tau^m} \cp(J)^d)Ê\ar[ur]^{\psi_2}&
}
\end{equation}
Consider first the map
$$\varphi_1: \pi_1
(S^1\times(\partial{B^3}-\{\mbox{two points} Ê \}))\to
\pi_1((X-\Sigma)^d-S^1\times (B^3-I)).$$ Note that $\pi_1
(S^1\times(\partial{B^3}-\{\mbox{two points} Ê \}))$ is $\Z^2$
generated by $[S^1]$ and $[\mu]$; we claim that their images under $\varphi_1$ are
trivial. To see this, note that $\varphi_1([\mu])\in
\pi_1((X-{\Sigma})^d)$ is the meridian $[\mu_{\tilde\Sigma}]$ of the
lifted surface of $\Sigma$ and so it projects to
$[\mu_{\Sigma}^d]\in \pi_1(X-\Sigma)$, which is trivial. Similarly,
$\varphi_1([S^1])$ is sent to $[\alpha]\in \pi_1(X-\Sigma)$ and so
it is trivial as well.

Now, consider
\begin{equation}\label{unbranchcover2}
 \pi_1(S^1\times_{\tilde\tau^m}
\cp(J)^d)/im(\varphi_2)\cong \la\pi_1(S^1\times_{\tilde\tau^m}
(S^3,J)^d)\mid [S^1]=1\ra .
\end{equation}

In Plotnick's paper ~\cite{plotnick:fibered}, he constructed a
knotted 2-sphere $A(J)$ in a homotopy $4$-sphere $\Omega$ that is
given by
$$(\Omega,A(J))=  P\cup_{A} S^1\times E(J)$$
where $P$ is a plumbing of two copies of $S^2\times D^2$ and $A$ is
expressed by a matrix form according to a certain basis
$\{e_1,e_2,e_3\}$ on $H_1 (\partial P)$. When our assumption $(m,d)=
1$ is given, there are $\gamma$ and $\beta$ such that
$d\gamma+m\beta=1$. If we perform Plotnick's construction along the
following gluing map

\begin{equation*}\label{matrix5}
   A=\begin{pmatrix}
      m&d&0\\
      -\gamma&\beta&0\\
      0&0&1
\end{pmatrix}
\end{equation*}then
the induced homotopy sphere $\Omega$ is smoothly $S^4$ and so we get
a knotted 2-sphere $A(J)$ in $S^4$. Moreover, the complement of
$A(J)$ in $S^4$ is fibred over $S^1$. In fact, the fiber is the
$d$-fold branched cover of $J$ and the monodromy is $\tilde\tau^m$
described in ~\eqref{lifttau}. So, we observe that the presentation
~\eqref{unbranchcover2} is the knot group of $A(J)$ in $S^4$ with
the relation $[S^1]=1$, which is indeed the meridian of $A(J)$ in
Plotnick's construction. So, $\pi_1(S^1\times_{\tilde\tau^m}
\cp(J)^d)//im(\varphi_2)$ is trivial. This shows
$$\pi_1((X - \Sigma_{J}(m))^d)\cong\pi_1((X-\Sigma)^d-S^1\times (B^3-I))\cong\pi_1((X-\Sigma)^d).$$

Now we need to check that this implies that $\pi_1(X -
\Sigma_{J}(m))\cong G$. Denoting $\pi_1((X-\Sigma)^d)$ by $H$, if $(m,d)=1$ then we have an short exact
sequence
\begin{equation}\label{hurewicz2}
 0\rightarrow H\rightarrow
\pi_1(X-\Sigma_J(m))\stackrel{hurew.}\rightarrow\Z/d\rightarrow 0.
\end{equation}
We are assuming that the meridian of $\Sigma$ has order $d$ in $\pi_1(X -\Sigma)$.  It is easy to check that this still holds when we remove the rim torus $\alpha \times \mu_{\Sigma}$, so that the meridian $\mu_{\Sigma_J(m)}$ has order $d$ as well.  Thus the sequence~\eqref{hurewicz2} splits, and the action of $\Z/d$ is again by conjugation by $\mu_{\Sigma_J(m)}$ on $H$.  Keeping track of this conjugation in the isomorphism described above shows that the sequences~\eqref{hurewicz} and~\eqref{hurewicz2} yield the same extension.
\end{proof}
Iterating this construction, we obtain
\begin{corollary}\label{iteratedtwist}
Consider a surface $\Sigma \subset X$ with $\pi_1(X - \Sigma) \cong \Z/m_1$, and a sequence of integers $m_2,\ldots,m_n$ such that $(m_1,m_i) = 1$ for all $i>1$.  Then the knot group of $\Sigma_{K_1,\ldots,K_n}(m_1,\ldots,m_n)$ is isomorphic to that of $\Sigma_{K_1}(m_1)$.
\end{corollary}
We remark that iterated twisted rim surgery can be done in a single operation, as follows.  Suppose that $K_1,\ldots,K_n$ are knots in $S^3$, and that integers $m_1,\ldots,m_n$ are given. Then the exterior $E(K_1\# \cdots \#K_n)$ is contains the exteriors $E(K_i)$ in a standard way, bounded by incompressible tori.  Performing the twist maps $\tau^{m_i}$ along these tori gives a diffeomorphism $\mathcal{T}$ of $E(K_1\# \cdots \#K_n)$ which gives rise to a new surface knot $(X,\Sigma_{K_1,\ldots,K_n}(\mathcal{T}))$ as in~\eqref{twistrimsurgery}.  This is the same as doing $m_i$-twist rim surgeries along the knots $K_i$ in any sequence.
\subsection{$1$-twist rim surgery}\label{1twistsec}
A simpler variation of the previous argument shows that $1$-twist rim surgery often preserves the surface knot group.

\begin{proposition}\label{1twistfundagp}
Suppose that $\alpha \subset \Sigma$ is an embedded curve that has a null-homotopic pushoff into $X - \Sigma$.  Then for any knot $K$, the surface $\Sigma_K(1)$ obtained by $1$-twist surgery along the rim torus parallel to $\alpha$ has $\pi_1(X - \Sigma_K(1)) \cong \pi_1(X - \Sigma).$
\end{proposition}

\begin{proof}
The van Kampen theorem for the decomposition $X-\Sigma_K(1)$ shows the following diagram;
$$
\xymatrix@C=-2pc{
&\pi_1(X-\Sigma-S^1\times (B^3-I))Ê\ar[dr]^{\psi_1}&\\
{\pi_1 (S^1\times(\partial{B^3}-\{\mbox{two points}\}))} Ê\ar[ur]^{\varphi_1} \ar[dr]^{\varphi_2}
&&  \pi_1(X - \Sigma_{K}(1))\\
& \pi_1(S^1\times_{\tau}(\cp(K)) \ar[ur]^{\psi_2}&
}
$$

In the diagram, consider the generators $[S^1]$ and $[\mu]$ in $\pi_1 (S^1\times(\partial{B^3}-\{\mbox{two points} Ê
\}))$. Note that $\varphi_1[\mu]$ is the meridian $[\mu_\Sigma]$ in $\pi_1(X-\Sigma-S^1\times (B^3-I))\cong\pi_1(X-\Sigma)$ and $\varphi_1[S^1]$ is $[\alpha]$ which is trivial. So, the presentation for $\pi_1(X - \Sigma_{K}(1))$ is
\begin{equation*}\label{1twistfundgpeq}
 \la \pi_1(X-\Sigma) *\pi_1(S^1\times_{\tau}
(\cp(K))) \mid [S^1]=1, \mu_\Sigma=\mu_K, \mu_K^{-1}\beta\mu_K=\beta\  , \forall\beta\in\pi_1(\cp(K))\ra ,
\end{equation*}
that is isomorphic to $\pi_1(X-\Sigma)$.
\end{proof}
In section~\ref{sec:1twistscob}, we will show that $1$-twist rim surgery does not change the $s$-cobordism class of the knot $\Sigma$, and hence the resulting surface is topologically unknotted for good fundamental groups.   In that section, we will make use of the observation, referring to the decomposition above, that the image of $\pi_1(S^1\times_{\tau}(\cp(K))$ in $ \pi_1(X - \Sigma_{K}(1))$ is the cyclic subgroup generated by the meridian of $ \Sigma_{K}(1)$.

In Section~\ref{smoothsec} we will use this result to get infinitely many smoothly knotted surfaces with a given fundamental group.

\section{Symplectic tori with arbitrary knot group}\label{symplecticsec}
In this section, we discuss the question of when a given finitely presented group $G$ is the fundamental group of $X - S$, where $X$ is a simply-connected symplectic $4$-manifold, and $S$ is a symplectic surface.  Note that $S$ being symplectic implies that  $[S]$  is non-trivial in $H_2(X;\R)$, which in turn implies that $H_1(X-S;\Z)$ is finite, in fact isomorphic to $\Z/d$ where $d$ is the divisibility of $[S] \in H_2(X;\Z)$. Thus, the following conditions are necessary for $G $ to be isomorphic to $\pi_1(X - S)$:
\begin{equation}\tag{$K_d$}\label{Kd}
H_1(G) = \Z/d \ \text{for some $d$, and}\  \exists\, \gamma\in G\ \text{such\ that}\  G/\langle \gamma \rangle = \{1\}.
\end{equation}
We show that these topological conditions are sufficient for the existence of a symplectic surface.  The technique is a relative version of Gompf's construction of symplectic manifolds with arbitrary fundamental group.  We have not made any effort to be efficient in this construction, with regard to keeping $\chi(M)$ or $\sigma(M)$ (or some combination thereof) small. It seems likely that this could be achieved for specific groups, following~\cite{baldridge-kirk:symplectic-pi1}.
\begin{theorem}\label{sympd}
If $G$ satisfies conditions~\eqref{Kd}, then there is a simply-connected symplectic $4$-manifold $M$ containing a symplectically embedded surface $S$ with $\pi_1(M-S) \cong G$.
\end{theorem}
\begin{proof}
Take a finite presentation of $G$:
$$
G = \la x_1,\ldots,x_l\, | \, r_1,\ldots r_m\ra
$$
Since $H_1(G) \cong \Z/d$, there are elements $a_i,\ b_i \in G$ with  $\gamma^d = \prod_i^n [a_i,b_i]$, where $\gamma$ is the group element in~\eqref{Kd}. Write $a_i$ and $b_i$ as words in the generators $x_j$, symbolically $a_i=v_i(x_1,\ldots,x_l)$ and $b_i = w_i(x_1,\ldots,x_l)$.   Similarly, write $\gamma$ as a word $w(x_1,\ldots,x_l)$.   By construction, the equation $w^d =  \prod_i^n [v_i,w_i]$ in the free group generated by the $\{x_i\}$  is a consequence of the relations $\{r_j\}$.

Consider a surface $\Sigma_1$ of genus $l +n$ with one boundary component $\eta'$, containing a standard symplectic basis of curves $\{x_1,y_1,\ldots,x_l,y_l,a_1,b_1,\ldots,a_n,b_n\}$.  Let $P$ denote a $d$-punctured disc $D^2 - (\delta_1\cup\cdots\cup\delta_d)$, with boundary $\partial P = \eta \bigcup \cup_i \partial (\delta_i)$.    Write $\Sigma_d = \Sigma_1 \cup_{\eta = \eta'} P$, and $\Sigma$ for $\Sigma_d$ with the disks $\delta_i$ glued back in.  This is illustrated, for $d=2$, in Figure~\ref{sigma} below.
\begin{figure}[!ht]
\centering
\hspace{1.4in}\includegraphics[scale=.6]{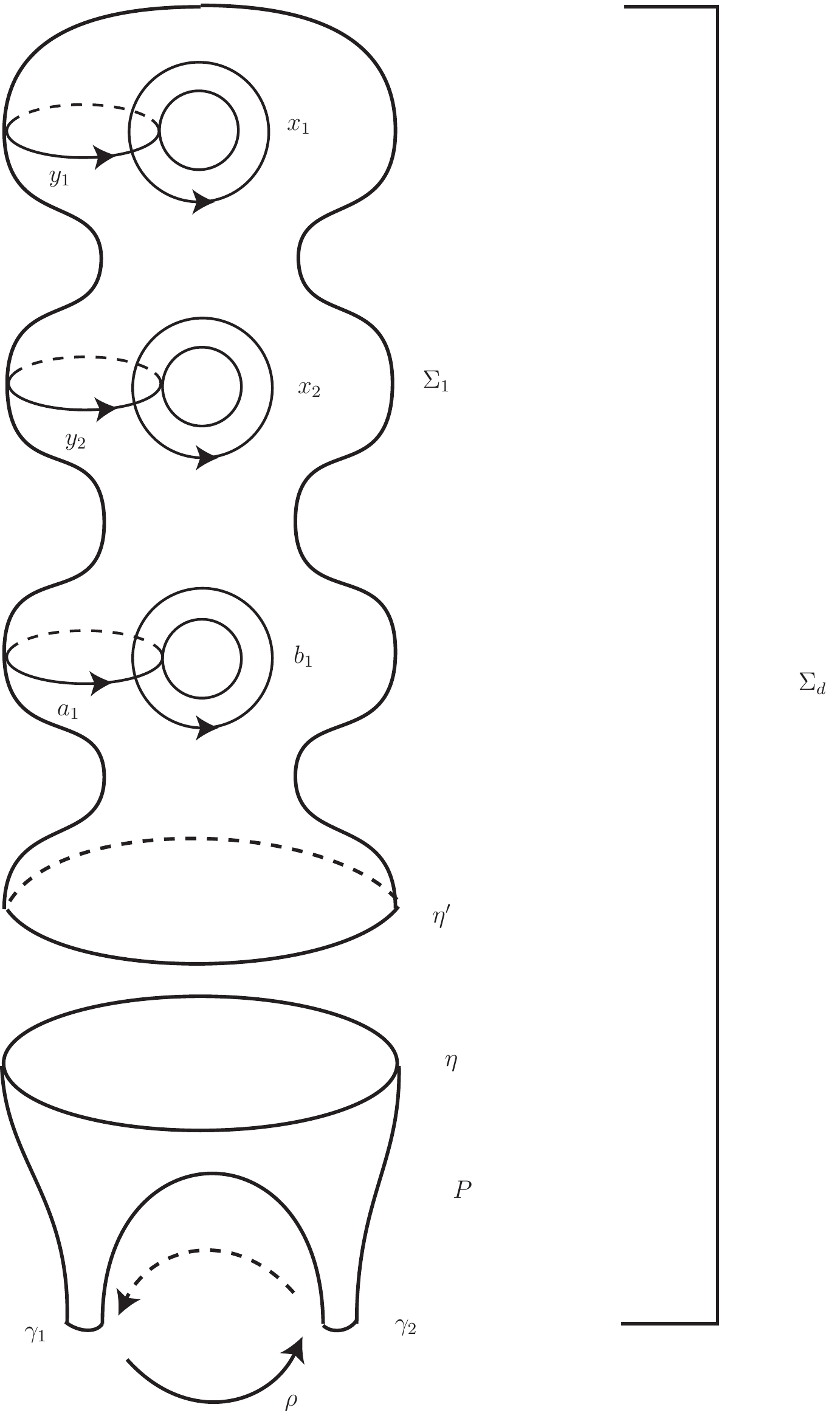}
\caption{}
\label{sigma}
\end{figure}

The disks $\delta_j$ should be cyclically arranged around a circle in $\text{int}(D)$ as shown below in Figure~\ref{planar}.   The boundaries of the $\delta_j$, oriented counterclockwise, together with the indicated base paths, will be denoted $\gamma_j \in \pi_1(P,1)$.   With the given orientations and base paths, $\eta =\gamma_d \cdot \gamma_{d-1} \cdots \gamma_1$.  Consider a diffeomorphism $\rho:D \to D$ that is the identity on $\partial D$, and permutes the $\delta_j$ cyclically.
The effect of $\rho$ on $\pi_1(P,1)$ is given by $\rho_*(\gamma_1) = \gamma_2$, \ldots, $\rho_*(\gamma_{d-1}) = \gamma_d$, and $\rho_*(\gamma_d) = \eta\gamma_1\eta^{-1}$.  Extend $\rho$ by the identity on $\Sigma_1$ so that it becomes a diffeomorphism on $\Sigma_d$.
\begin{figure}[!h]
\centering
\includegraphics[scale=.70]{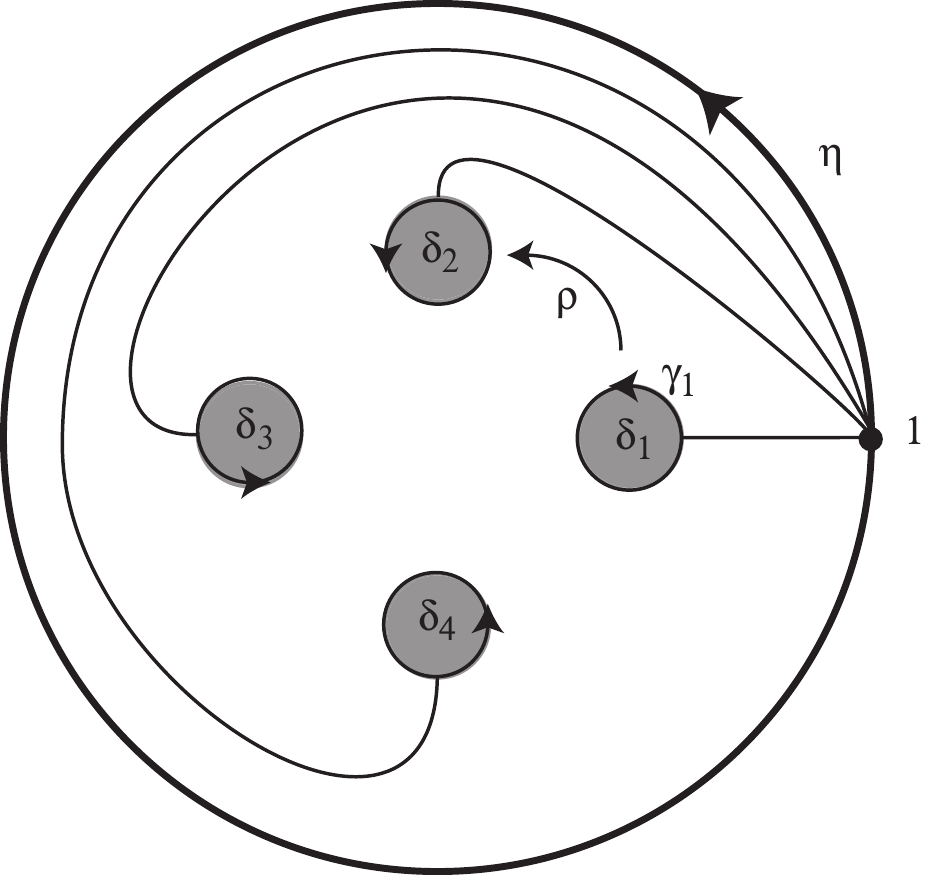}
\caption{}
\label{planar}
\end{figure}

Form the manifold $X = S^1 \times S^1 \times \Sigma$, with a product symplectic structure.  We will call the first circle factor $\alpha$ and the second one $\beta$ and write $T = \alpha \times \beta$.  Note that the surface $T^2 \times 0$ where $0$ is the center of $D^2$ is a symplectic submanifold of $X$. In the solid torus $\beta \times D^2$ there is a braid that forms a $(d,1)$ torus knot meeting $D$ in the centers of the disks $\delta_j$. Let $S_d$ be the product of $\alpha$ with this braid; it is straightforward to check that the symplectic structure on $X$ can be arranged so that $S_d$ is symplectic.

The complement $X_d$ of $S_d$ is equal to the product of the circle $\alpha$ with the mapping torus $S^1 \times_\rho \Sigma_d$; we will identify the circles transverse to $\Sigma_1 \subset \Sigma_d$ with the $\beta$ circles.   Another way to observe the embedding of $X_d$ in $X$ is to note that if $\rho$ is extended over the disks $\delta_j$, then it is actually isotopic to the identity map of $\Sigma$, and so the mapping torus becomes a product manifold.

The fundamental group of $X_d$ is generated by $\alpha, \beta$ and
$$
\{x_i,y_i,a_j,b_j,\gamma_k\}\ \text{for}\ \,\ i=1 \ldots l,\ j=1\ldots n,\ k=1\ldots d.
$$
with relations that $\alpha$ commutes with everything, $\beta$ commutes with $x_i,y_i,a_j$ and $b_j$,
\begin{align}
\gamma_k^\beta & = \gamma_{k+1} \ \text{for}\ k=1\ldots d-1\\
 \gamma_d^\beta &= \gamma_{1}^\eta, \ \text{and}\\
\label{comm}
\prod_{i=1}^l [x_i,y_i] \prod_{j=1}^n [a_j,b_j] & =\gamma_d \cdot \gamma_{d-1} \cdots \gamma_1
\end{align}

Choose, for $j\ge 1$, immersed curves $\epsilon_j \subset \Sigma_1$ for  representing the homotopy classes (in $\Sigma_1$) of
\begin{equation}\label{rels}
y_1,\ldots,y_l,r_1,\ldots,r_m,\  \text{and}\ a_1^{-1}v_1,\ldots,a_n^{-1}v_n,b_1^{-1}w_1,b_n^{-1}w_n.
\end{equation}
and a curve $\epsilon_0$ in $\Sigma_d$ representing $\gamma_1^{-1}w$.    Following
Gompf~\cite{gompf:symplectic}, replace $\Sigma_d$ by its connected sum with many copies of $T^2$ and the $\epsilon_j$ by their connected sum with curves running over these tori so that $\alpha \times \epsilon_j$ can be arranged to be embedded and symplectic.    (The collection of curves $\{\epsilon_j\}$ has been enlarged in this process to include the generators $\pi_1$ of each torus added on.)   Note that the connected sums can all be arranged to take place in $\Sigma_1 \subset \Sigma_d$, and so the diffeomorphism $\rho$ extends to the new $\Sigma_d$.

Now do the symplectic sum of $X$ with copies of the elliptic surface $E(1)$, where a fiber $F$ of $E(1)$ is identified with the each of the tori $\alpha \times \epsilon_j$.  Do one further symplectic sum where $F$ is identified with a copy of the torus $\alpha \times \beta$.   Write $M$ for the result of all of these fiber sums with $X$, with $M_d$, the complement of $S_d$, being the same fiber sums with $X_d$.

Since the fiber in $E(1)$ has simply-connected complement,  van Kampen's theorem implies that each fiber sum kills the precisely elements of the fundamental group of the torus in $X$ or $X_d$.   Let us compute the fundamental group of $X_d$.   Note that after relations killing the fundamental group elements $\alpha$, $\beta$ and those listed in~\eqref{rels} are imposed, then only the generators
$\{\gamma_1,x_1,\ldots,x_l\}$ are needed, and the relation~\eqref{comm} reduces to
$$
 \prod_{j=1}^n [v_j,w_j]  =\gamma_1^d.
$$
Doing the final fiber sum along $\alpha \times \epsilon_0$ makes $\gamma_1$ a word in the $x_i$, with this relation automatically satisfied.   Thus $\pi_1(M_d)$ is generated by the $\{x_i\}$ with relations $\{ r \}$, and is thus isomorphic to $G$.  The fundamental group of $M$ is trivial, because we kill the element $\gamma=\gamma_1$, which by hypothesis normally generates $G$.
\end{proof}
There are two special property of the surface constructed in the above proof: it is a torus and has $0$-self-intersection.  It is straightforward, when $d=1$, to modify the construction to produce surfaces of arbitrary positive genus $g$.  In that case, instead of taking a product with a torus to form $X_1$, simply take the product with a surface $F_g$ of genus $g$.   To kill the extra fundamental group introduced in this way, one needs to take a symplectic sum of $X_1$ along a $F_g$ with a symplectic manifold $Y$ containing a copy of $F_g$ with simply-connected complement; such are easily found.  It seems a little more difficult to find such surfaces for $d>1$.

Finding a symplectic sphere has a rather different character.   The fundamental group of a sphere with non-zero self-intersection $m$ in a simply-connected $4$-manifold satisfies extra relations, because the element $\gamma$ in conditions~\eqref{Kd} satisfies $\gamma^m =1$.  There are certainly many groups that satisfy conditions~\eqref{Kd} but have no elements of finite order (for example $d$-framed surgery on most knots in $S^3$).   So to get a group satisfying conditions~\eqref{Kd} we would want the sphere to have trivial normal bundle.  However, smoothly embedded spheres  with trivial normal bundles are relatively rare in symplectic manifolds, and so we conjecture that there are some groups that simply cannot be realized.

In general, if a surface $\Sigma$ of genus $g$ is embedded in a $4$-manifold with self-intersection $k$, note that its divisibility $d$ must divide $k$.   By considering a pushoff of $\Sigma$, one sees that the  group of $\Sigma$ satisfies conditions~\eqref{Kd} with the extra proviso that the element $\gamma$ may be chosen so that $\gamma^k$ is a product of at most $g$ commutators.   In principle, this places some extra restriction on the group $G$, but this seems hard to work with because there may be many choices for $\gamma$.

If we do not care whether the ambient manifold is symplectic, then it is easy to find embedded surfaces of any genus with arbitrary group satisfying conditions~\eqref{Kd}.
\begin{proposition}\label{sphere}
Let $G$ be a finitely presented group satisfying~\eqref{Kd}.  Then $G = \pi_1(X-S)$ for an embedded $2$-sphere $S$ in a smooth simply-connected $4$-manifold $X$.
\end{proposition}
Note that by adding on handles to $S$, we get surfaces of arbitrary genus.
\begin{proof}
Construct a handlebody $Y$ with $1$-handles and $2$-handles corresponding to the generators and relations of a presentation of $G$.  Represent $\gamma$ by an embedded circle in $\partial Y$, and let $Z$ be $Y$ together with a $2$-handle attached along $\gamma$ (with arbitrary framing).  Then $Z$ is simply-connected, and contains a properly embedded disc $\Delta$ (the cocore of this handle) such that $Z - \Delta$ deformation retracts onto $Y$.  In particular, $\pi_1(Z - \Delta) \cong G$.  Let $X$ be the double of $Z$, and take $S $ to be the double of the disk $\Delta$.  Note that $\pi_1\partial Y \to \pi_1 Y$ is surjective, which implies that $\pi_1(X - S) \cong G$.
\end{proof}

\section{Topological classification}\label{topclass}
The iterated twisted rim surgery construction of section~\ref{iterated} and the construction of symplectic surfaces in section~\ref{symplecticsec} (combined with $1$-twist rim surgery as in Proposition~\ref{1twistfundagp}) give large families of surface knots with the same knot group.  This section will treat the topological classification of these knots, with the smooth classification considered in the next section.

First we discuss the iterated twist rim surgery construction, starting with a surface $(X,\Sigma)$.   We make the same hypotheses as in Section~\ref{iterated} on the group $G$ of $\Sigma$ and the curve
$\alpha$ that determines the rim torus.  We perform a twisted rim surgery (with twisting $m$ such that $(m,d) = 1$) to obtain a new knot $(X,\Sigma_J(m))$, with the same group as $\Sigma$.   In the case that the knot groups were cyclic, we used in our earlier paper~\cite{kim-ruberman:surfaces} a computation in surgery theory to show that all such knots are pairwise homeomorphic.  This relied on the vanishing of the Wall groups $L_5^h(\Z[\Z_d])$ and $L_5^s(\Z[\Z_d])$, which does not hold for arbitrary fundamental groups.  So we return to the method of~\cite{kim:surfaces}, and work under the assumption that  $J$ is a ribbon knot. To compare the surfaces $\Sigma$ and $\Sigma_J(m)$
topologically in $X$, we will construct a relative $h$-cobordism
between their exteriors as in ~\cite{kim:surfaces}; a further condition on the Alexander polynomial of $J$ will ensure the vanishing of the torsion.  Note that for simplicity of notation, our computations in the first $3$ sections had to do with knot and surface \emph{complements}; for this section we will work with the \emph{exteriors} so we can use the relative s-cobordism theorem.  To this end, we will write $\nun{J}$ for an open tubular neighborhood of $J$.

Let us briefly review the construction from~\cite{kim:surfaces}, to which we refer for further details. If $J$ is a ribbon knot then there is a concordance $A$ in $S^3\times I$
between $J$ and an unknot $O$, such that the map $\pi_1(S^3 -J) \to \pi_1(S^3 \times I -A)$ is a surjection.   The twist map $\tau$ on $(S^3,J)$ extends to a self diffeomorphism with the same name on $(S^3\times I, A)$ as follows. On the collar of $\partial\nu(A)\cong
A\times{\partial{D^2}}\times{I}$,
$$
\tau(x\times{e^{i\theta}}\times{t}) =x\times{e^{i(\theta + 2\pi
t)}}\times{t} \quad \mbox{for} \quad
x\times{e^{i\theta}}\times{t}\in A\times{\partial{D^2}}\times{I}
$$
and otherwise, $\tau$ is the identity. Note that the restrictions
$\tau$ to $S^3\times \{0\}$ and $S^3\times \{1\}$ are the twist maps
$\tau_O$ and $\tau_J$ generated by $O$ and $J$.

Write $S^3=B^3_{+}\cup B^3_{-}$ and let $J=J_+\cup J_-$ where 
$J_+ = B^3_+ \cap J$ and $J_-$ is an unknotted arc in $B^3_-$. We
obtain a restricted concordance between the arcs $J_+$ and $O_+$ by
taking out $B^3_{-}\times I$ from $(S^3\times I, A)$ and then denote
the concordance by $A_{+}$ in $B^3_{+}\times I$. Using $\tau$
restricted to $(B^3_{+}\times I, A_{+})$, we obtain a new pair
$(X\times I, (\Sigma\times I)_A(m))$ by taking out the neighborhood
of the curve $\alpha\subset\Sigma$ in $X\times I$ and gluing back
$(B^3_{+}\times I, A_{+})$ along $\tau^m$. Explicitly, we write
$$(X\times I, (\Sigma\times I)_A(m))= X\times I-S^1\times
(B^3\times I, I\times I)\cup S^1\times_{\tau^m}(B^3\times I, {A_+}).$$

Note that in this construction, $X\times 1= (X,\Sigma_J(m))$ and
$X\times 0=(X,\Sigma)$. Consider the exterior $X\times
I-\nun{(\Sigma\times I)_A(m)}$, denoted by $W$, which provides a homology
cobordism between $X-\nun{\Sigma}$ and $X-\nun{\Sigma_J(m)}$ (see the proof of Proposition 4.3 in ~\cite{kim:surfaces}).   Like all of the cobordisms we will consider in this section, $W$ is a product along the boundary.  Let $M_0=X-\nun{\Sigma}$ and
$M_1=X-\nun{\Sigma_J(m)}$. From the decomposition of $(X\times I,
(\Sigma\times I)_A(m))$, we write $W$ as
\begin{equation}\label{cobordism}
W=(X-(S^1\times B^3)-\nun{\Sigma})\times I\cup S^1\times_{\tau^m}(B^3\times
I-\nun{A_+}).
\end{equation}

We assert that $W$ is an $h$-cobordism; the first step is to show that
$\pi_1(W)=G$. Then we show, for the universal covers
$\widetilde W$ and $\widetilde{M_1}$ of $W$ and $M_1$ respectively, that
$H_*(\widetilde W,\widetilde{M_1})$ is trivial. The Whitehead Theorem shows that the induced map by
inclusion, $i: M_1\to W$ is a homotopy equivalence.

\begin{lemma}\label{cobo:fundamentalgp}
$\pi_1(W)$ is isomorphic to G.
\end{lemma}
\begin{proof}
We show first that the fundamental group of the $d$-fold cover of $W$ is
isomorphic to that of the $d$-fold cover of $M_1$ and then deduce that $\pi_1(W)\cong \pi_1(M_1)$.

Since the homology class of $\alpha$ is trivial in $W$ and $M_1$, we decompose the $d$-fold covers $W^d$ and $M_1^d$ from~\eqref{cobordism} as follows.
\begin{align}
\label{M:2covering}
M_1^d &=(X-S^1\times B^3-\nun{\Sigma})^d\cup
S^1\times_{\tilde{\tau}_J^m}(B^3-\nun{J_+})^d
\\
\label{W:2covering}
W^d &=(X-S^1\times B^3-\nun{\Sigma})^d\times I\cup
S^1\times_{\tilde{\tau}^m}(B^3\times I-\nun{A_+})^d
\end{align}

Applying the van Kampen theorem to these decompositions, we can compare the two diagrams:


\begin{equation*}
\xymatrix{
&\pi_1((X-S^1\times B^3-\nun{\Sigma})^d) \ar@{-}@<-6ex>[dd]^{i_2}Ê\ar[dr]^{\psi_1}&\\
{\pi_1 (S^1\times(\partial{B^3}-\{\mbox{two points}\}))} \ar[ddd]^{i_1}Ê\ar[ur]^{\varphi_1} \ar[dr]^{\varphi_2}
&&  \pi_1(M_1^d)\ar[ddd]^{i_4}\\
& \pi_1(S^1\times_{\tilde\tau_J^m} \cp(J)^d)\ar@<-6ex>[d]Ê\ar@<-2.0ex>[ddd]^>>>>>>>>{i_3}\ar[ur]^{\psi_2}&
\\
&\pi_1((X-S^1\times B^3-\nun{\Sigma})^d\times I)\ar[dr]^{\psi_1'}&\\
{\pi_1 (S^1\times(\partial{B^3}-\{\mbox{two points}\}\times I))} Ê\ar[ur]^{\varphi_1'} \ar[dr]^{\varphi_2'}
&&  \pi_1(W^d)\\
& \pi_1(S^1\times_{\tilde\tau_A^m} (B^3\times
I-\nun{A_+})^d)Ê\ar[ur]^{\psi_2'}&
}
\end{equation*}
Here, $i_1$, $i_2$, $i_3$, and $i_4$ are the maps induced by inclusions.
Obviously, $i_{1}$ and $i_{2}$ are isomorphisms. Moreover, $i_3$ is
onto since $J$ is a ribbon knot and so the induced map $i_4$ is
onto.  Recall from Proposition~\ref{fundamentalgp} that $\pi_1(M_1^d)$
is isomorphic to $ \pi_1 ((X-S^1\times B^3-\nun{\Sigma})^d).$  Our
claim is that $\pi_1(W^d)$ is isomorphic to $\pi_1 ((X-S^1\times
B^3-\nun{\Sigma})^d\times I)$ and so it is isomorphic to
$\pi_1(M_1^d)$. This will then imply that $\pi_1(W)$ is isomorphic
to $\pi_1(M_1)$.

If we consider the argument of Proposition~\ref{fundamentalgp},
$\varphi_1$ is the zero map and the quotient
$\pi_1(S^1\times_{\tilde\tau_J^m} \cp(J)^d)//im(\varphi_2)$ is
trivial. In the diagram of $\pi_1(W^d)$, similarly $\varphi'_1$ is
the zero map and the quotient $\pi_1(S^1\times_{\tilde\tau_J^m}
\cp(J)^d)//im(\varphi_2')$ is also trivial by the commutativity and
surjectivity of $i_3$. This shows that $\pi_1(W^d)$ is isomorphic to
$\pi_1 ((X-S^1\times B^3-\nun{\Sigma})^d\times I)$.

Thus, we have the following commutative diagram:
$$
\begin{array}{ccccccccc}
0 \! & \longrightarrow & \! \pi_1(M_1^d) \! & \longrightarrow & \!
\pi_1(M_1) \! & \longrightarrow & \! H_1(M_1) \! & \longrightarrow
& \! 0\vspace{5 pt} \\
  & & \cong\!\big\downarrow & & \!\big\downarrow & &
\cong\!\big\downarrow &
&  \\
0 \! & \longrightarrow & \! \pi_1(W^d) \! & \longrightarrow & \!
\pi_1(W) \! & \longrightarrow & \! H_1(W) \! & \longrightarrow
& \! 0\\
\end{array}
$$

Note that $ \pi_1(M_1)$ maps to $\pi_1(W)$ surjectively and some simple diagram-chasing shows that it is an isomorphism.
\end{proof}

\begin{proposition}\label{prop:cobo}
If $J$ is a ribbon knot and the homology of $d$-fold cover
$(S^3-J)^d$, $H_1((S^3-J)^d)\cong\Z$ with $(m,d)=1$ then there
exists an $h$-cobordism $W$ between $M_0=X-\nun{\Sigma}$ and
$M_1=X-\nun{\Sigma_J(m)}$ rel $\partial$.
\end{proposition}
\begin{proof}
We will show that for the universal coverings of $W$ and
$M_1$ denoted by $\widetilde W$ and $\widetilde{M_1}$ respectively,
 $H_*(\widetilde W,\widetilde{M_1})$ is trivial. If follows by the Whitehead theorem that the inclusion $i :M_1\to W$ is a homotopy
equivalence.

We first consider the $d$-fold covers of $W$ and $M_1$ associated to
$\pi_1(W)\to H_1(W)=\Z/d$ and $\pi_1(M_1)\to H_1(M_1)=\Z/d$.   As before, denote by $H$ the groups $\pi_1(W^d)\cong\pi_1(M_1^d)$. Note that the universal covers of $W^d$ and $M_1^d$ are the universal
covers of $W$ and $M_1$.  We will denote the preimage, under the
universal cover, of a subset $S$ of $W^d$ or $M_1^d$ by $S^H$, and
refer to this as the $H$-cover of $S$.   Then the universal covers of $W$ and $M$ decompose into the preimages of the pieces in the 
decompositions ~\eqref{W:2covering} and ~\eqref{M:2covering} of their $d$-fold covers $W^d$ and $M^d$:
\begin{equation}\label{universalcover:W}
\widetilde W=((X-S^1\times B^3-\nun{\Sigma})^{d})^H\times I\cup
(S^1\times_{\tilde{\tau}^m}(B^3\times I-\nun{A_+})^d)^H
\end{equation}
and
\begin{equation}\label{universalcover:M_1}
\widetilde M_1=((X-S^1\times B^3-\nun{\Sigma})^{d})^H\cup
(S^1\times_{\tilde{\tau}_J^m}(B^3-\nun{J_+})^d)^H.
\end{equation}

In order to describe the $H$-cover of
$S^1\times_{\tilde{\tau}_J^m}(B^3-\nun{J_+})^d$, we consider the
inclusion-induced map
$\pi_1(S^1\times_{\tilde{\tau}_J^m}(B^3-\nun{J_+})^d)\to
\pi_1(M_1^d)$.
In the diagram ~\eqref{d-diagram} induced by the van Kampen theorem to the decomposition of $M_1^d$ in ~\eqref{branchcover2}, the argument of Proposition~\ref{fundamentalgp} shows that $\psi_2:\pi_1(S^1\times_{\tilde{\tau}_J^m}(B^3-\nun{J_+})^d)\to
\pi_1(M_1^d)$ is trivial.


This means that the $H$-cover of
$S^1\times_{\tilde{\tau}_J^m}(B^3-\nun{J_+})^d$ is the disjoint
union of
 copies of $S^1\times_{\tilde{\tau}_J^m}(B^3-\nun{J_+})^d$, indexed by elements of $H$.
A similar argument shows that the $H$-cover of
$S^1\times_{\tilde{\tau}^m}(B^3\times I-\nun{A_+})^d$ is the
disjoint union of copies of $S^1\times_{\tilde{\tau}^m}(B^3\times
I-\nun{A_+})^d$ indexed by elements of $H$ as well. So, by excision,
the relative homology for the pair $(\widetilde W,\widetilde{M_1})$
takes the following simple form:
$$
H_*(\widetilde W,\widetilde{M_1})\cong
\overset{H}{\oplus} H_*((B^3\times I-\nun{A_+})^d, (B^3-\nun{J_+})^d))
$$
where $\overset{H}{\oplus}$ means a direct sum indexed by the elements of $H$.

Our assumption $H_1((S^3-J)^d)\cong\Z$ implies that $H_*((B^3\times
I-\nun{A_+})^d, (B^3-\nun{J_+})^d)$ is trivial and so the result
follows.
\end{proof}

Now we shall compute the Whitehead torsion of the pair $(W,M_1)$. If
this is zero then we would obtain an $s$-cobordism $W$ between $M_0=X-\nun{\Sigma}$ and
$M_1=X-\nun{\Sigma_J(m)}$ rel $\partial$ that is topologically trivial if $G$ is a good group.

The homotopy equivalence $i: M_1 \to W$ induces a well-defined
Whitehead torsion $\tau(W,M_1)$; for the pair of universal covers
$(\widetilde W,\widetilde{M_1})$, the chain complex $C(\widetilde
W,\widetilde{M_1})$ over $\Z[\pi]$ is acyclic where $\pi=\pi_1(W)=G$
and so we have the torsion $\tau(W,M_1)\in Wh(\pi)$. Computing the
Whitehead torsion of each component of the decomposition for
$(W,M_1)$ according to the sum theorem, we may obtain $\tau(W,M_1)$;
recall that the decomposition of the pair $(W,M_1)=(X\times
I-\nun{(\Sigma\times I )_{A}(m)}, X-\nun{\Sigma_J(m)})$  is
\begin{equation*}
\begin{split}
((X-S^1\times B^3-\nun{\Sigma})
&\times I\cup S^1\times_{\tau^m}(B^3\times I- \nun{A_{+}}),\\
& X-S^1\times B^3-\nun{\Sigma}\cup
S^1\times_{{\tau}_J^m}(B^3-\nun{J_+})).
\end{split}
\end{equation*}
If we rewrite this as
\begin{equation}\label{decomp:cobo}
\begin{split}
((X-S^1\times B^3-\nun{\Sigma})\times I, & X-S^1\times B^3-\nun{\Sigma})\ \cup \\
& (S^1\times_{\tau^m}(B^3\times I-\nun{ A_{+}}),
S^1\times_{{\tau}_J^m}(B^3-\nun{J_+})),
\end{split}
\end{equation}
then we can observe that the Whitehead torsion of the first
component pair $((X-S^1\times B^3-\nun{\Sigma})\times I, X-S^1\times
B^3-\nun{\Sigma})$ is zero. However, the Whitehead torsion
$\tau((S^1\times_{\tau^m}(B^3\times I-\nun{ A_{+}}),
S^1\times_{{\tau}_J^m}(B^3-\nun{J_+}))$ is not defined since
$S^1\times_{{\tau}_J^m}(B^3-\nun{J_+})$ may not be a deformation
retract of $S^1\times_{\tau^m}(B^3\times I- \nun{A_{+}})$.

Thus, we now consider the Reidemeister torsion of $(W,M_1)$
according to the identity homomorphism
$id:\mathbf{Z}[\pi]\longrightarrow \mathbf{Z}[\pi]$. Note that the
relation
$$\tau^{id}(W,M_1)=id_*\tau(W,M_1)$$
shows that if the Reidemeister torsion associated to the identity is
trivial then the Whitehead torsion is zero. So, our claim is that
the Reidemeister torsion $\tau^{id}(W,M_1)$, denoted simply by
$\tau(W,M_1)$, according to the coefficient $\mathbf{Z}[\pi]$ is
trivial.

By the gluing theorem, we compute the Reidemeister torsion of each
component in the decomposition of $(W,M_1)$ to get $\tau(W,M_1)$.
Indeed, we need to check if the Reidemeister torsion of each
component is well defined over $\mathbf{Z}[\pi]$.

\begin{theorem}\label{homeo}
If $J$ is a ribbon knot and the homology of the $d$-fold cover of
$S^3-J$, $H_1((S^3-J)^d)\cong\Z$ with $(m,d)=1$, and $G$ is a good
group, then $(X,\Sigma)$ is pairwise homeomorphic to
$(X,\Sigma_J(m))$.
\end{theorem}
\begin{proof}
According to Proposition ~\ref{prop:cobo} and the above argument, it
is sufficient to show that the Reidemeister torsion $\tau(W,M_1)$ is
trivial. We denote by
$K_1\cup K_2$ the union of each component of the decomposition ~\eqref{cobordism} of $W$, and likewise $L_1\cup L_2$ for $M_1$. Note that from ~\eqref{universalcover:W}, ~\eqref{universalcover:M_1} in Proposition ~\ref{prop:cobo},
the universal cover of $W$ is the union of $H$-covers of the $d$-fold covers of each component and the same form works for $M_1$. Extending the notation, we write $\widetilde W=\widetilde
K_1\cup \widetilde K_2$ and
$\widetilde M_1=\widetilde L_1\cup
\widetilde L_2$,
where for each $\alpha=1,2$, $\widetilde K_{\alpha}$ is the preimage of $K_{\alpha}$ under the universal cover of $W$
associated to the inclusion induced map $i_{\alpha *}: \pi_1
K_{\alpha}\to \pi_1 W=G$ and likewise, $\widetilde L_{\alpha}$ for  $L_{\alpha}$.

Then we have the Mayer-Vietoris sequence for the pair $(\widetilde W,
\widetilde M_1)=(\widetilde K_1, \widetilde L_1)\cup (\widetilde
K_2, \widetilde L_2)$ and $(\widetilde K_0, \widetilde
L_0)=(\widetilde K_1, \widetilde L_1)\cap (\widetilde K_2,
\widetilde L_2)$.
If the Reidemeister torsion of each component in the decomposition
of $(W,M_1)$ associated to the inclusion induced morphism
$i_{\alpha*}:\Z[\pi_1(K_{\alpha})]\to \Z[\pi_1(W)]$ is well defined
then the Mayer-Vietoris sequence and the multiplicativity of torsion shows
\begin{equation*}
\begin{split}
\tau(W,M_1) \cdot\tau^{i_0}(K_0,L_0) = \tau^{i_1}(K_1,L_1) \cdot
\tau^{i_2}(K_2,L_2).
\end{split}
\end{equation*}
Since the chain complexes $C_*(\widetilde K_0, \widetilde L_0)$ and
$C_*(\widetilde K_1, \widetilde L_1)$ are obviously acyclic, their
torsion $\tau^{i_0}(K_0,L_0)$ and $ \tau^{i_1}(K_1,L_1)$ are well
defined and moreover they are trivial. So, we only need to compute 
$\tau^{i_2}(K_2,L_2)=\tau^{i_2}(S^1\times_{\tau^m}(B^3\times
I-\nun{ A_{+}}), S^1\times_{{\tau}_J^m}(B^3-\nun{J_+}))$ to get $\tau(W,M_1)$. In order to define the
torsion $\tau^{i_2}(K_2,L_2)$, we need to check if the chain
complex $C_*(\widetilde K_2, \widetilde L_2)$ is acyclic. In Proposition~\ref{prop:cobo}, 
$\widetilde K_2$ is the disjoint union of $H$-copies of
$S^1\times_{\tilde{\tau}^m}(B^3\times I-\nun{A_+})^{d}$. Similarly, $\widetilde L_2$ is the disjoint union of $H$- copies of $S^1\times_{{\tau}_J^m}(B^3-\nun{J_+})^d$ and so we write 

$$(\widetilde K_2, \widetilde L_2)=(\overset{H}{\coprod} S^1\times_{\tilde{\tau}^m}(B^3\times I-\nun{A_+})^{d},
\overset{H}{\coprod}S^1\times_{\tilde{\tau}_J^m}(B^3-\nun{J_+})^d).$$

So, the chain complex is the form of 
$$C_*(\widetilde K_2, \widetilde L_2)\cong
\overset{H}{\oplus} C_*(S^1\times_{\tilde{\tau}^m}(B^3\times
I-\nun{A_+})^{d}, S^1\times_{\tilde{\tau}_J^m}(B^3-\nun{J_+})^d).$$ Since
$(S^3-J)^d$ is a homology circle, the relative homology $H_*((B^3\times I-\nun{A_+})^{d},
(B^3-\nun{J_+})^d)$ is trivial and so $C_*(\widetilde K_2,\widetilde L_2)$
is acyclic. Now, in the pair $(K_2, L_2)=(S^1\times_{{\tau}^m}(B^3\times I-\nun{A_+}),
S^1\times_{{\tau}_J^m}(B^3-\nun{J_+}))$, we consider it as a relative smooth fiber bundle over $S^1$ with the fiber $(B^3\times I-
\nun{A_{+}}, B^3-\nun{J_+})$:
$$
(B^3\times I- A_{+}, B^3-\nun{J_+})\hookrightarrow
(S^1\times_{\tau^m}(B^3\times I- A_{+}),
S^1\times_{{\tau}_J^m}(B^3-\nun{J_+}))\longrightarrow S^1
$$

For simplicity, we denote its relative fiber by $(F,F_0)$ and so we write the
pair of covers $(\widetilde K_2,\widetilde L_2)$  as
$(\overset{H}{\coprod} S^1\times_{\tilde{\tau}^m}\widetilde F,
\overset{H}{\coprod} S^1\times_{{\tilde\tau}^m}\widetilde F_{0})$, where $(\widetilde F,\widetilde F_0)$ is the pair of the $d$-fold covers associated to the inclusions. Using the same techniques as in Proposition 4.4 of  ~\cite{kim:surfaces}, we have the following exact sequence for
$(\overset{H}{\coprod} S^1\times_{\tilde{\tau}^m}\widetilde F,
\overset{H}{\coprod} S^1\times_{{\tilde\tau}^m}\widetilde F_{0})$;



\begin{equation*}
\begin{split}
0&\longrightarrow \overset{H}{\oplus} (C_*(\widetilde F, \widetilde
F_{0})\oplus C_*(\widetilde
F, \widetilde F_{0})) \\
& \longrightarrow \overset{H}{\oplus} (C_*([0,1/2]\times \widetilde
F, [0,1/2]\times \widetilde F_0)\oplus C_*([1/2,1]\times \widetilde
F, [1/2,1]\times
\widetilde F_0))\\
&\longrightarrow \overset{H}{\oplus} (C_*(S^1\times_{\tilde
\tau^m}\widetilde F, S^1\times_{\tilde \tau^m}\widetilde
F_0))\longrightarrow 0.
\end{split}
\end{equation*}

By the assumption that $(\widetilde F,\widetilde F_0)$ is
homologically trivial, it follows that if
$j:\Z[\pi_1(F)]\longrightarrow \Z[\pi_1(W)]$ denotes the morphism
induced by inclusion then the torsion $\tau^{j}(F, F_0)$ is defined.
From the above short exact sequence and the multiplicativity of the
torsion we deduce that $\tau^{i_2}
(S^1\times_{\tau^m}F,S^1\times_{\tau^m}F_0)=0$ implies that $\tau^{i_2}(K_2,L_2) $ is also trivial and thus
 the Reidemeister torsion $\tau(W,M_1)\in K_1(\Z[G])/\pm G$ is
trivial.
\end{proof}

\subsection{Topological triviality for $1$-twist rim surgery}\label{sec:1twistscob}
In section~\ref{1twistsec}, we showed that a $1$-twist rim surgery does not change the fundamental group.  In this section, we show that the surface $\Sigma_K(1)$ produced by such a surgery is standard up to $s$-cobordism.   The construction is similar to that in the previous section, but we allow $K$ to be slice (rather than ribbon), and impose no hypothesis on the cyclic coverings of $S^3 - K$.
\begin{theorem}\label{1twistscob}  Suppose that $\alpha \subset \Sigma$ is an embedded curve that has a null-homotopic pushoff into $X - \Sigma$.  Then for any slice knot $K$, the surface $\Sigma_K(1)$ obtained by $1$-twist rim surgery along $\alpha$ is $s$-cobordant to $\Sigma$.
\end{theorem}
The homology computations requires a preliminary lemma.  Suppose that $\pi: \tY \to Y$ is an infinite cyclic cover, and that $T:\tY \to \tY$ generates the group of covering translations.    Note that for any $k$, the quotient $Y^k = \tY/\la T^k \ra$ is a $k$-fold cyclic cover of $Y$, and that $T$ descends to a generator of the covering translations of $Y^k$.
\begin{lemma}\label{zcover}
The mapping torus $S^1 \times_T \tY$ is homeomorphic to $\R \times Y$.  Moreover, $S^1 \times_T Y^k$ is homeomorphic to $S^1 \times Y$.
\end{lemma}
\begin{proof}
Our convention is that the mapping torus is given by $\R \times \tY$, modulo the relation $(x,y) \sim (x-1,Ty)$.   The map $(x,y) \mapsto  \pi(y)$ descends to an $\R$ bundle over $Y$, and the same map descends to an $S^1$ bundle $S^1 \times_T Y^k \to Y$.  To see that these bundles are trivial, note that the covers $\tY \to Y$ and $Y^k \to Y$ are induced from the standard infinite and finite cyclic covers $\R \to S^1$ and $S^1 \to S^1$ by a map $f: Y \to S^1$.    This implies that the $\R$-bundle $S^1 \times_T \tY \to Y$ is induced from $S^1 \times_T \R \to S^1$ by the same map, and likewise for the circle bundles.

Hence, it suffices to consider the case where $Y = S^1= \R/\Z$, with $\tY = \R$.  Then the first bundle is trivialized by the isomorphism $\R \times S^1 \to S^1 \times_T \R$ that takes $(r,[y])$ to $[y+r,-y]$, where the brackets $[\ ]$ denote equivalence classes.  This trivialization descends to a trivialization of the circle bundle as well.
 \end{proof}

\begin{proof}[Proof of Theorem~\ref{1twistscob}]
The proof uses the same technique as in Proposition~\ref{prop:cobo} and Theorem~\ref{homeo}, so we will be brief.  A concordance of $K$ to the unknot produces a cobordism between $X -\nun{\Sigma}$ and $X- \nun{\Sigma_K(1)}$, as described at the beginning of this section.  The isomorphism
$G =\pi_1(X- \nun{\Sigma}) \cong \pi_1(X- \nun{\Sigma_K(1)})$ established in Proposition~\ref{1twistfundagp} works as well to calculate that the fundamental group of $W = X\times I-\nun{(\Sigma\times I)_A(m)}$ is also $G$.  So the remaining points are to show the vanishing of the relative homology groups of the universal covers of
$(W,X- \nun{\Sigma})$, and the Whitehead torsion.

As observed just after the proof of Proposition~\ref{1twistfundagp}, the image of $\pi_1(S^1\times_{\tau}(B^3 -\nun{K_+}))$ in $ \pi_1(X - \nun{\Sigma_{K}(1)})$ is the cyclic subgroup generated by the meridian of $ \Sigma_{K}(1)$; the same is true for the image of  $\pi_1(S^1\times_{\tau}(B^3\times I -\nun{A_+})$ in $G$.  It follows that in the universal cover, the preimage of $S^1\times_{\tau}(B^3 -\nun{K_+})$ is a union of its finite or infinite cyclic covers, the order of the meridian in $\pi_1(X -\Sigma_{K}(1))$ determining the order of the covering.  The same is true for the preimage of $B^3\times I -\nun{A_+}$.   Note that these coverings are mapping tori, as in Lemma~\ref{zcover}, where in place of the covering transformation $T$, we have a lift of the twist map $\tau$.   But the lift $\tilde\tau$ defined in~\eqref{lifttau} is isotopic to the covering transformation, so we can apply Lemma~\ref{zcover} to compute the homology of these covering spaces.  It follows that each $(B^3\times I -\nun{A_+},B^3 -\nun{K_+})$ lifts to a relative homology cobordism, and hence (by the Mayer-Vietoris argument in Proposition~\ref{prop:cobo}) that $W$ is a relative h-cobordism.

The torsion calculation in Theorem~\ref{homeo} depends only on the vanishing of the relative homology of $(B^3\times I -\nun{A_+},B^3 -\nun{K_+})$, with coefficients in the group ring $\Z[G]$ induced by the inclusion of these spaces into $W$.  But the argument in the preceding paragraph implies this vanishing, so that the torsion is trivial.
\end{proof}
\begin{corollary}\label{1twisthomeo}
With the hypotheses of Theorem~\ref{1twistscob}, if the group $G$ is good, then the knots $\Sigma_K(1)$ and $\Sigma$ are topologically equivalent.
\end{corollary}
\section{Smooth classification}\label{smoothsec}
To distinguish, in the smooth category, the knots that we have constructed, we make use of the results of Fintushel and Stern~\cite{fs:surfaces}.   They start with a surface $\Sigma$ such that $\Sigma \cdot \Sigma = 0$ and $(X,\Sigma)$ is an SW-pair, and show that for knots $K_1,\ K_2$, the equality $(X,\Sigma_{K_1}) \cong (X,\Sigma_{K_2})$ implies that the coefficients of $\Delta_{K_1}$ coincide with those of $\Delta_{K_2}$ (including multiplicities).   This result (described in the addendum to the original paper) is proved using the gluing theory in~\cite{kronheimer-mrowka:monopole}.   Note that by blowing up, one can convert a surface with positive self-intersection into one with $0$ self-intersection; symplectic surfaces of {\em negative} self-intersection are treated in a recent preprint of T.~Mark~\cite{mark:hf-surfaces}.

The extra $m$-twist does not affect the gluing theorem, and so (with hypotheses as above) the surfaces $(X,\Sigma_{K_1}(m)) $ and $(X,\Sigma_{K_2}(m))$ are distinguished smoothly if the coefficients of their Alexander polynomials form distinct sets.  Here is a sample result that one gets by combining these observations with the constructions of section~\ref{iterated} and the topological classification results in section~\ref{topclass}.
\begin{theorem}\label{s2s2}
For any odd number $p$, there are infinitely many  topologically equivalent but  smoothly inequivalent knots in $S^2 \times S^2$ with dihedral knot group $D_{2p}$ with homology class $(2,2)$ in the obvious basis for $H_2(S^2 \times S^2)$.
\end{theorem}
\begin{proof}
Let $X = S^2 \times S^2$. Choose a complex curve $\Sigma$ in the homology class $(2,2)$; this will be a torus of square $8$, and have group $\Z/2$.  For any odd $p$ and $q$ relatively prime to $p$, the surface $(X,\Sigma_{K_{p,q}}(2))$ has group $D_{2p}$, by Lemma~\ref{branchcover}.   Let $J$ be any knot with non-trivial Alexander polynomial but with determinant $1$.   For any positive $n$, let $J_n$ be the ribbon knot $\#_n (J \# -J)$.  By Theorem~\ref{homeo}, the knots $(X,\Sigma_{K_{p,q},J_n}(2,3))$ are all topologically equivalent.  On the other hand, these surfaces are all distinguished smoothly because the coefficient lists for the Alexander polynomials of the $J_n$ are distinct.
\end{proof}

In a different direction, the construction of symplectic surfaces also gives rise to families of smoothly distinct surfaces.
\begin{theorem}\label{infsymp}
Let $G$ be a group satisfying condition~\eqref{Kd}.  Then there is a simply-connected symplectic $4$-manifold $M$ containing a symplectically embedded surface $S$, and infinitely many smoothly embedded surfaces $S_n$ in the same homology class with $\pi_1(M-S_n) \cong G$.    If $G$ is a good group then these surfaces can be taken to be topologically equivalent.
\end{theorem}
\begin{proof}
Start with the symplectic surface $S$ with group $G$ provided by Theorem~\ref{sympd}; note that since $S$ is symplectic and has $0$ self-intersection, $(M,S)$ is an SW-pair.   In the construction of the surface $S$, we performed fiber sums multiple times, including a fiber sum to kill the generators $\alpha$ and $\beta$ of the fundamental group of the torus $T$.  It follows readily that $\alpha$, pushed into the complement of $S$, is null-homotopic in the complement of $S$.   Choose a sequence of knots $J_n$ as in the previous theorem, and do $1$-twist rim surgeries to create new surfaces $(M,S_{J_n}(1))$, all of which have group $G$.   These are smoothly distinct, as before.

By Theorem~\ref{1twistscob}, if the knots $J_n$ are slice, then the knots $(M,S_{J_n}(1))$ are all $s$-cobordant.  If the group $G$ is good, then the knots are topologically equivalent.
\end{proof}


\begin{thebibliography}{10}

\bibitem{baldridge-kirk:symplectic-pi1}
Scott Baldridge and Paul Kirk, \emph{On symplectic 4-manifolds with prescribed
  fundamental group}, Comment. Math. Helv. \textbf{82} (2007), no.~4, 845--875.
  \MR{MR2341842}

\bibitem{cao-zhu:ricci}
Huai-Dong Cao and Xi-Ping Zhu, \emph{A complete proof of the {P}oincar\'e and
  geometrization conjectures---application of the {H}amilton-{P}erelman theory
  of the {R}icci flow}, Asian J. Math. \textbf{10} (2006), no.~2, 165--492.
  \MR{MR2233789}

\bibitem{fs:surfaces}
Ronald Fintushel and Ronald~J. Stern, \emph{Surfaces in {$4$}-manifolds}, Math.
  Res. Lett. \textbf{4} (1997), no.~6, 907--914, Addendum:
  http://arxiv.org/abs/math.GT/0511707. \MR{MR1492129 (98k:57047)}

\bibitem{freedman-teichner:subexponential}
Michael H.~Freedman and Peter~Teichner, {\em {$4$}-manifold topology. {I}.
  {S}ubexponential groups}, Invent. Math., {\bf 122} (1995), 509--529.

\bibitem{freedman-quinn}
Michael H.~Freedman and Frank Quinn, \emph{Topology of $4$-manifolds}, Princeton
  University Press, Princeton, N.J., 1990.

\bibitem{gompf:symplectic}
Robert~E. Gompf, \emph{A new construction of symplectic manifolds}, Ann. of
  Math. (2) \textbf{142} (1995), no.~3, 527--595. \MR{MR1356781 (96j:57025)}

\bibitem{kim:surfaces}
Hee~Jung Kim, \emph{Modifying surfaces in 4-manifolds by twist spinning}, Geom.
  Topol. \textbf{10} (2006), 27--56 (electronic). \MR{MR2207789}

\bibitem{kim-ruberman:surfaces}
Hee~Jung Kim and Daniel Ruberman, \emph{Topological triviality of smoothly
  knotted surfaces in 4-manifolds}, {Trans.\ A.M.S.}, to appear;
  http://www.arxiv.org/abs/math.GT/0610204, 2006.

\bibitem{kronheimer-mrowka:monopole}
Peter B. Kronheimer and Tomasz S. Mrowka, \emph{{Monopoles and Three-Manifolds}},
  Cambridge University Press, Cambridge, UK, 2007.

\bibitem{krushkal-quinn:subexponential}
Vyacheslav S. Krushkal and Frank Quinn, {\em Subexponential groups in 4-manifold
  topology}, Geom. Topol., {\bf 4} (2000), 407--430 (electronic).

\bibitem{mark:hf-surfaces}
Thomas~E. Mark, \emph{Knotted surfaces in 4-manifolds},
  http://arxiv.org/abs/0801.4367, 2008.

\bibitem{morgan-tian:poincare}
John Morgan and Gang Tian, \emph{Ricci flow and the {P}oincar\'e conjecture},
  Clay Mathematics Monographs, vol.~3, American Mathematical Society,
  Providence, RI, 2007. \MR{MR2334563}

\bibitem{plotnick:fibered}
Steven~P. Plotnick, \emph{Fibered knots in {$S\sp 4$}---twisting, spinning,
  rolling, surgery, and branching}, Four-manifold theory (Durham, N.H., 1982),
  Contemp. Math., vol.~35, Amer. Math. Soc., Providence, RI, 1984,
  pp.~437--459. \MR{MR780592 (87a:57021)}

\bibitem{plotnick-suciu:spherical}
Steven~P. Plotnick and Alexander~I. Suciu, \emph{Fibered knots and spherical
  space forms}, J. London Math. Soc. (2) \textbf{35} (1987), no.~3, 514--526.
  \MR{MR889373 (88f:57038)}

\end{thebibliography}
\def\cprime{$'$}
\providecommand{\bysame}{\leavevmode\hbox to3em{\hrulefill}\thinspace}
\providecommand{\MR}{\relax\ifhmode\unskip\space\fi MR }
\providecommand{\MRhref}[2]{%
  \href{http://www.ams.org/mathscinet-getitem?mr=#1}{#2}
}
\providecommand{\href}[2]{#2}

 \end{document}